\documentclass[10pt]{amsart}
\usepackage{ amsthm, amssymb, amsmath,url}

\theoremstyle{definition}
\newtheorem{notation}[subsection]{Notation}

\theoremstyle{theorem}
\newtheorem{thmA}{Theorem}
\newtheorem{theorem}{Theorem}[section]

\newtheorem{lem}[theorem]{Lemma}
\newtheorem{prop}[theorem]{Proposition}

\theoremstyle{definition}

\newtheorem{example}[theorem]{Example}

\theoremstyle{remark}

\RequirePackage{lineno}

\newcommand{\La}{\Lambda}
\newcommand{\vx}{\operatorname{vx}}

\newcommand{\w}{\widetilde}
\newcommand{\wh}{\widehat}
\renewcommand{\o}{\overline}
\newcommand{\la}{\lambda}
\newcommand{\Cl}{\operatorname{cc}}
\newcommand{\Bl}{\operatorname{Bl}}
\newcommand{\Irr}{\operatorname {Irr}}
\newcommand{\Syl}{\operatorname {Syl}}

\newcommand{\Cent}{\operatorname C}

\newcommand{\Z}{\operatorname Z}

\newcommand{\bl}{\operatorname{bl}}

\newcommand{\NNN}{\operatorname {N}}

\newcommand{\TT}{\mathbb T}
\newcommand{\IBr}{\operatorname {IBr}}

\newcommand{\CC}{\mathbb C}

\def\spann<#1>{\left\langle#1\right\rangle}

\newcommand{\calO}{\mathcal O}

\newcommand{\gmid}{\,\Big|\,}

\newcommand{\phip}{{\phi'}}
\newcommand{\wphip}{{\w\phi'}}
\newcommand{\wphi}{{\w\phi}}
\newcommand{\wchi}{{\w\chi}}

\newcommand{\chip}{{\chi'}}
\newcommand{\wchip}{{\w\chi'}}

\newcommand{\thetap}{{\theta'}}
\newcommand{\wthetap}{{\w\theta'}}
\newcommand{\wtheta}{{\w\theta}}

\renewcommand{\labelenumi}{(\alph{enumi})}

\renewcommand{\theenumi}{\thesubsection{}(\alph{enumi})}

\newcommand{\enumalph}{\renewcommand{\labelenumi}{(\alph{enumi})} \renewcommand{\theenumi}{\thesubsection(\alph{enumi})}}
\renewcommand{\theenumi}{\thesubsection(\alph{enumi})}

\numberwithin{equation}{section}

\title{Clifford theory of characters in induced blocks}
\author[]{ Shigeo Koshitani and Britta Sp\"ath}
\thanks{
\newline\indent
The first author was supported by the Japan Society for Promotion
of Science (JSPS), Grant-in-Aid for Scientific Research (C)23540007,
2011--2014.
The second author has been supported by the Deutsche 
Forschungsgemeinschaft, SPP 1388 and by the ERC Advanced Grant 291512.}
\address{Department of Mathematics, Graduate School of Science, 
Chiba University, 1-33 Yayoi-cho, Inage-ku,
Chiba, 263-8522, Japan}
\email{koshitan@math.s.chiba-u.ac.jp}
\address{FB Mathematik, TU Kaiserslautern, Postfach 3049, 
67653 Kaiserslautern, Germany.}
\email{spaeth@mathematik.uni-kl.de}
\dedicatory{\normalsize
 Dedicated to the memory of Masafumi Murai}
\date{\today}

\begin{document}
\setpagewiselinenumbers
\modulolinenumbers[5]
\linenumbers

\begin{abstract}
We present a new criterion to predict if a character of a finite 
group extends. 
Let $G$ be a finite group and $p$ a prime.
For $N\lhd G$, we consider $p$-blocks $b$ and $b'$ 
of $N$ and $\NNN_N(D)$, respectively, with $(b')^N=b$, 
where $D$ is a defect group of $b'$.
Under the assumption that $G$ coincides with a normal subgroup $G[b]$
of $G$, which was introduced by Dade early 1970's,
we give a character correspondence 
between the sets of all irreducible constituents 
of $\phi^G$ and those of $(\phi')^{\NNN_G(D)}$ 
where $\phi$ and $\phi'$ are irreducible Brauer characters in $b$
and $b'$, respectively. 
This implies a sort of generalization of the theorem of Harris-Kn\"orr.
An important tool is the existence of certain extensions that also
helps in checking the inductive Alperin-McKay and 
inductive Blockwise Alperin Weight conditions, due to the second author.
\end{abstract}
\maketitle

{\small{\tableofcontents{}}}

\section{Introduction}

It is a well-known and difficult problem 
to predict the Clifford theory of characters of finite groups, 
although there are numbers of well-known criteria 
that involve finite group theory of the underlying groups. 
Here we give a different approach that makes assumptions 
on associated blocks and their properties. 
The work of \cite{NavarroSpaeth} suggests that 
the Clifford theories of height zero characters 
in Brauer corresponding blocks should have some similarities. 
More precisely the notion of block-isomorphic-character-triples 
has been introduced there, and it is conjectured that 
there is a bijection between height zero characters of blocks 
such that associated character triples are block isomorphic, 
since it is consequence of Theorem 7.1 of \cite{NavarroSpaeth} that assumes the inductive AM condition. 
We do not investigate block-isomorphic-character-triples here in detail, 
since it is a quite technical concept, but we give some evidence 
for this and an analogue for Brauer characters in a limited situation. 

This situation turns out in \cite{KoshitaniSpaeth13} 
to be important in checking 
the inductive Alperin-McKay 
(inductive AM) 
and inductive Blockwise-Alperin-Weight 
(inductive BAW) 
conditions for blocks with cyclic defect groups. 

Throughout this paper $G$ is always a finite group, $p$ is a prime,
and a block is a $p$-block.
For $N\lhd G$ and a block $b$ of $N$, Dade 
introduced in \cite{DadeBlockExtensions} a normal subgroup $G[b]$ of the stabilizer $G_b$ of $b$ in $G$, and established its remarkable properties. 
Later Murai gave a character theoretic interpretation of this group 
in \cite{Murai_Dade}. We analyse the situation where $G=G[b]$. 
For an irreducible character $\w\chi$ of $G$ we investigate the block of $G$ the character $\w\chi$ belongs to, and we denote it by $\bl(\w\chi)$. 
We write $H \leq G$ if $H$ is a subgroup of $G$.
If $H\leq G$ and $B'$ is a block of $H$,
we denote by $(B')^G$ the induced block (block induction) 
whenever it is defined. 
Let $\IBr(G)$ be the set of all irreducible Brauer characters 
of $G$, $\IBr(b)$ the set of all irreducible Brauer characters 
of a block $b$ of $N$
and for $\phi\in\IBr(N)$ we denote by $\IBr(G\mid \phi)$ 
the set of all irreducible constituents of the induced character $\phi^G$. 

\begin{thmA}\label{thmA}
Let $G$ be a finite group, and
let $N\lhd G$, $H\leq G$ and $M:=N\cap H$. 
Let $b'\in\Bl(M)$ be a block of $M$ that has a defect group 
$D$ with $\Cent_G(D)\subseteq H$. Assume that $G=G[b]$ for 
$b:=(b')^N$, where $G[b]$ is the group as defined in Notation \ref{Def_G[b]}. 
Then for every $\phi \in\IBr(b)$ and every $\phip\in\IBr(b')$ 
there is a bijection 
 \[ \Lambda: \IBr(G\mid \phi)\longrightarrow \IBr(H\mid \phip),\]
 such that $\bl(\Lambda(\rho))^G=\bl(\rho)$ 
for every $\rho\in\IBr(G\mid \phi)$. 
Further $\rho \in \IBr(G)$ is an extension of $\phi$ 
if and only if $\Lambda(\rho)$ is an extension of $\phip$.
\end{thmA}

A similar related statement in the context 
of Green correspondence is given in \cite{Dade84}.
Note also that here the defect group of $b'$ and $b$ might differ. 
In the case where the defect groups coincide the statement follows from 
Section 12 of \cite{DadeBlockExtensions}. 
Theorem \ref{thmA} is proven by 
constructing various extensions of characters 
that belong to specific blocks. As a consequence 
we obtain a sort of generalization of the Theorem of Harris-Kn\"orr
\cite{HarrisKnoerr}.

\begin{thmA}\label{thm_gen_HK}
Let $G$ be a finite group, and
let $N\lhd G$, $H\leq G$ and $M:= N\cap H$. 
Let $b'\in\Bl(M)$ be a block of $M$ that has a defect group $D$ 
with $\Cent_G(D)\subseteq H$. For $b:=(b')^N$ 
the map from $\Bl(H\mid b')$ to $\Bl(G\mid b)$ 
given by $B'\mapsto (B')^G$ is well-defined and surjective,
where $\Bl(G\mid b)$ is the set of all blocks of $G$ covering $b$.
\end{thmA} 

Actually this work grew out of investigations on the inductive 
blockwise Alperin weight conditions 
(inductive BAW conditions) 
and the inductive Alperin-McKay conditions 
(inductive AM conditions) 
for blocks with cyclic defect groups that lead to \cite{KoshitaniSpaeth13}. 
Those inductive conditions are defined 
in the context of the reduction theorems 
of the Alperin-McKay and Alperin weight conjecture.
Theorem \ref{thm1_3} is a powerful tool in the investigations. 
We suspect our main results here should simplify the 
checking of the inductive conditions in general and hope 
that it gives new insights why they should hold. 
When $B$ is a block of $G$, we write $\Irr(G)$ and $\Irr(B)$
for the sets of irreducible ordinary characters of $G$ and
those which are in $B$, respectively.

\begin{thmA}\label{thm1_3}
Let $G$ be a finite group, and let $N\lhd G$, $H \leq G$ and
$M:= N\cap H$. Let $b' \in\Bl(M)$ be a block of $M$ with defect
group $D$ such that $\Cent_G(D)\subseteq H$, and let $b := (b')^N$.
Assume further that $G = G[b]$.
\begin{enumerate}
\item 
{\sf (Ordinary characters)}
  \begin{enumerate}
  \item[(1)]  
If $\chi' \in\Irr(b')$ extends to a character $\w\chi' \in\Irr(H)$,
then there exists a character $\chi\in\Irr(b)$ of height zero
which extends to a character $\w\chi\in\Irr(G)$ and which satisfies
\[ 
(*) \qquad
\bl\Big( (\wchip)_{J\cap H}\Big)^{J}=\bl(\w\chi_J)\text{ for every }J 
\text{ with }N\leq J\leq G.
\]
  \item[(2)]
If $\chi\in\Irr(b)$ extends to a character $\w\chi\in\Irr(G)$, then
there exists a character $\chi'\in\Irr(b')$ of height zero
which extends to a character $\w\chi'\in\Irr(H)$ and which satisfies $(*)$.
   \end{enumerate}
\item 
{\sf (Sylow $p$-subgroups)}
   \begin{enumerate}
   \item[(1)]
If $\chi'\in\Irr(b')$ extends to a character $\w\chi'\in\Irr(H)$
and if $\chi\in\Irr(b)$ extends to a subgroup $J_0$ of $G$
with $N \leq J_0 \leq G$ and $J_0/N \in\Syl_p(G/N)$,
then $\chi$ extends to a character $\w\chi\in\Irr(G)$ which satisfies 
$(*)$.
   \item[(2)] If $\chi\in\Irr(b)$ extends to a character $\w\chi\in\Irr(G)$
and if $\chi'\in\Irr(b')$ extends to $J_0\cap H$ for a subgroup
$J_0$ of $G$ with $N\leq J_0\leq G$ and $J_0/N\in\Syl_p(G/N)$,
then $\chi'$ extends to a character $\w\chi'\in\Irr(H)$ which 
satisfies $(*)$.
    \end{enumerate}
\item 
{\sf (Brauer characters)}
   \begin{enumerate}
   \item[(1)]
If $\phi'\in\IBr(b')$ extends to a character $\w\phi'\in\IBr(H)$,
then any $\phi\in\IBr(b)$ extends to a character $\w\phi\in\IBr(G)$
which satisfies
\[ (**) \qquad
\bl\Big( (\w\phi')_{J\cap H}\Big)^{J}=\bl(\w\phi_J)\text{ for every }J 
\text{ with }N\leq J\leq G.
\]
    \item[(2)]
If $\phi\in\IBr(b)$ extends to a character $\w\phi\in\IBr(G)$, then
any $\phi'\in\IBr(b')$ extends to a character $\w\phi'\in\IBr(H)$
which satisfies $(**)$.
    \end{enumerate}
\end{enumerate}
\end{thmA}

In this paper we introduce in Section \ref{sec_not} 
the general notation and some fundamental character theoretic tools. 
In Section \ref{sec3} we explore Dade's group $G[b]$ and 
deduce several criteria for the extensibility of characters,
and give a proof of Theorem \ref{thm1_3}.
We conclude the paper by constructing the character 
correspondence from Theorem \ref{thmA} and proving by Theorem \ref{thm_gen_HK} a generalization of the Harris-Kn\"orr theorem.

\section{Notation and general lemmas}\label{sec_not}
In this section we explain most of the used notation 
and give several lemmas which are useful for our main results. 

\begin{notation}[Characters and $p$-blocks]
For characters and blocks we use mainly the notation of 
\cite{NagaoTsushima} and \cite{Navarro}, respectively. 
Let $p$ be a prime. Let $(\mathcal K, \calO, k)$ be a 
$p$-modular system, that is "big enough" with respect to 
all finite groups occurring here. That is to say, 
$\calO$ is a complete discrete valuation ring of
rank one such that its quotient field $\mathcal K$ is 
of characteristic zero, 
and its residue field $k=\calO/\mathrm{rad}(\calO)$ is of
characteristic $p$, and that $\mathcal K$ and $k$ are
splitting fields for all finite groups occurring in this paper.
We denote the canonical epimorphism from $\calO$ to $k$ 
by $^*: \calO\rightarrow k$. 

In this paper $G$ is always a finite group.
For $g\in G$ we write $\Cl_G(g)$ for the conjugacy class of $G$
containing $g$. We denote by ${\mathrm{Cl}}(G)$ the set of
all conjugacy classes in $G$. For any subset $X$ of $G$
we denote by $X^+$ the sum $\sum_{x\in X} x$ in $\mathcal KG$
and in $kG$, respectively. 
We write $G^0$ for the set of all $p$-regular elements of $G$. 
For subsets $S, T \subseteq G$ we write $S\subseteq_G T$
if $S^g := g^{-1}Sg \subseteq T$ for some element $g \in G$.

We denote by $\Bl(G)$ the set of all ($p$-)blocks of $G$. 
If $N\lhd G$ and $b\in \Bl(N)$, 
we write $\Bl(G\mid b)$ for the set of all $p$-blocks of $G$ covering $b$.
We write $\Irr(G)$ and $\IBr(G)$ respectively for the sets of
all irreducible ordinary and Brauer characters of $G$. 
If $H \leq G$ and if $\chi$ and $\theta$ are characters of $G$ and $H$
respectively, then we denote by $\chi_H$ and $\theta^G$ 
the restriction of $\chi$ to $H$ and the induction of $\theta$ to $G$,
respectively. For the $p$-block $B \in \Bl(G)$ we write $\Irr(B)$
 and $\IBr(B)$ respectively
for the sets of all characters in $\Irr(G)$ and $\IBr(G)$ which
belong to $B$. 

For a character $\theta\in\IBr(G)\cup \Irr(G)$ 
we denote by $\bl(\theta)$ the $p$-block of $G$ which $\theta$ belongs to. 
If $N\lhd G$, we say that $\o B\in\Bl(G/N)$ is contained in $B$ 
and write $\o B\subseteq B$, 
if all irreducible characters of $\o B$ lift to characters of $B$, 
see \cite[p.198]{Navarro} or \cite[Theorem 5.8.2]{NagaoTsushima}. 
If $N \lhd G$ and $\phi \in \IBr(N)$, 
then we write $\IBr(G\mid\phi)$ for the set of 
all irreducible Brauer characters $\psi$ of $G$ 
such that $\phi$ is an irreducible constituent of $\psi_N$, 
see page 155 of \cite{Navarro}.

We denote by $\la_{B}: \Z(kG)\rightarrow k$ 
the central function (central character) 
associated with $B$, 
see page 48 of \cite{Navarro}. 
For $\phi\in\IBr(B)$ and $\chi\in\Irr(B)$ 
we write also $\la_\phi$ and $\la_\chi$ instead of $\la_B$. 

For a group $A$ acting on a set $X$ and for $X' \subseteq X$,
we denote by $A_{X'}$ the stabilizer of $X'$ in $A$.
In particular, if $N \lhd G$ and $b \in \Bl(N)$, we denote by
$G_b$ the stabilizer of $b$ in $G$.
Throughout this paper by a module we mean a finitely generated 
right module unless stated otherwise.
For an indecomposable $kG$-module $X$ we denote by $\vx(X)$
a vertex of $X$, which is a $p$-subgroup of $G$, 
see \cite[Definition before 4.3.4]{NagaoTsushima}.
For $kG$-modules $X$ and $Y$,
we write $Y|X$ if $Y$ is (isomorphic to) a direct summand of $X$. 

\end{notation}

In addition we also consider sometimes blocks as algebras. 

\begin{notation}[Blocks as bimodules]
Assume  $H\leq G$ and $B\in \Bl(G)$.
We can see $B$ also as a two-sided ideal of the group algebra $kG$ 
and hence it becomes a $k[G\times G]$-module
via $\beta{\cdot}(g_1, g_2) := {g_1}^{-1}\beta g_2$
for $\beta \in B$ and $g_1, g_2 \in G$. 
We denote by $B{\downarrow}_{H\times H}$ 
the $k[H\times H]$-module given by the restriction of $B$. 
We define the diagonal subgroup $\Delta H$ of $H\times H$ by 
$\Delta H := \{(h,h) \in H\times H\mid h\in H \}$.
\end{notation}

The following is 
related to the theorem of Harris-Kn\"orr 
\cite[(9.28)]{Navarro}, see also \cite{HarrisKnoerr}.

\begin{lem}\label{lem2_3} 
Let $H\leq G$ and $B\in\Bl(G)$ and $Q \leq H$ a $p$-subgroup 
with $Q\Cent_G(Q)\subseteq H$. 
\begin{enumerate}
\item We can write
\begin{align*}
 B{\downarrow}_{H\times H}
= { \begin{pmatrix}
  	\bigoplus  &  B'  \  \\
{\scriptstyle{B'\in\Bl(H)}}& \\
{\scriptstyle{\Delta Q \ \subseteq_{H\times H} \vx (B')}  }  & \  
\end{pmatrix}} 
&\bigoplus 
{ \begin{pmatrix}   \bigoplus  &  B' \   \\
{\scriptstyle{B'\in\Bl(H)}}& \\
{\scriptstyle{\Delta Q \ {\not\subseteq}_{H\times H} \ \vx (B')}}  & \  
\end{pmatrix}}
\\
&\bigoplus
\begin{pmatrix}
                    \bigoplus &   X  \  \\
{\scriptstyle{\mathrm{indec.} 
\ k[H\times H]\text{-}\mathrm{module} \ X \ \ \ }
} & \
   \\
{\scriptstyle{
X{\Big|}k[HgH] \ \text{for some }g \in G\backslash H}
} & \
   \end{pmatrix}.
\end{align*} 

\item \label{lem2_3b} Let $N\lhd G$. 
Assume there is $b\in\Bl(N)$ with $b \mid (B{\downarrow}_{N\times N})$. 
Let $M:=N\cap H$, and suppose that there is a block $b'\in\Bl(M)$ 
that has a defect group containing $Q$ and satisfies $(b')^N = b$. 
Then there is a block $B'\in\Bl(H \mid b')$ with $(B')^G = B$.
\end{enumerate}
\end{lem}

\begin{proof}
Part (a) follows easily by Lemma 5.10.9, and 
the proof of Theorem 5.10.12(i) of \cite{NagaoTsushima}. 
 
In the situation assumed in (b) we have that $(b')^N = b$, and hence 
$b'\mid (b{\downarrow}_{M\times M})$, 
see Theorem 5.10.11 of \cite{NagaoTsushima}. 
Hence, by the assumption $b\mid (B{\downarrow}_{N\times N})$, 
we have $b' \mid (B{\downarrow}_{M\times M})$. From part (a) 
we know the direct summands of $B{\downarrow}_{H\times H}$. 
Then one of the following three cases holds: 

\noindent
{\sf Case 1.} $b'\mid (B'{\downarrow}_{M\times M})$
for some $B'\in\Bl(H)$ with $\Delta Q \subseteq_{H\times H}\vx(B')$.

\noindent
{\sf Case 2.} $b'\mid (B'{\downarrow}_{M\times M})$
for some $B'\in\Bl(H)$ with $\Delta Q \not\subseteq_{H\times H}\vx(B')$.

\noindent
{\sf Case 3.} $b'\mid (X{\downarrow}_{M\times M})$
for some indecomposable $k[H\times H]$-module $X$ 
with $X\mid k[HgH]$ for some $g\in G\setminus H$.

Suppose that {\sf Case 2} occurs, namely,  
$b'\mid (B'{\downarrow}_{M\times M})$
for some $B'\in\Bl(H)$ with $\Delta Q \not\subseteq \vx(B')$. 
Then, Lemma 4.3.4(ii) of \cite{NagaoTsushima} implies that 
$\vx(b') \subseteq_{H\times H}\vx(B')$. Green's theorem 
\cite[Theorem 5.10.8]{NagaoTsushima} shows that 
$\vx(b') = \Delta D_1$, where $D_1$ is a defect group of $b'$. 
By the assumption $Q \subseteq D_1$, we have $\Delta Q \subseteq 
\Delta D_1 \subseteq_{H\times H}\vx(B')$, a contradiction. 
So, {\sf Case 2} can not occur.

Suppose that {\sf Case 3} occurs, that is, 
$b'\mid (X{\downarrow}_{M\times M})$
for some indecomposable $k[H\times H]$-module $X$ with $X\mid k[HgH]$ 
for some $g\in G\setminus H$. Recall that 
$\Cent_G(Q) \subseteq H$. The proof of Theorem 5.10.12(i) of 
\cite{NagaoTsushima} shows that $X \mid k[HgH]$ 
for $g\in H$, a contradiction.

This implies that {\sf Case 1} holds, namely, 
there is a block $B'\in\Bl(H)$ such that
$b' \mid (B'{\downarrow}_{M\times M})$ and 
$\Delta Q \subseteq \vx(B')$. Since $\vx(b') \subseteq_{H\times H}\vx(B')$ 
by Lemma 4.3.4(ii) of \cite{NagaoTsushima} 
\[\vx(b') \subseteq_{H\times H}\vx(B')= \Delta(D),\] 
where $D$ is a defect group of $B'$.
Hence, by the hypothesis,
$\Delta Q \subseteq \Delta D_1 = \vx(b')
 \subseteq_{H\times H} \vx(B') = \Delta D$.
So we may assume $Q \subseteq D$, and hence
$\Cent_G(D) \subseteq \Cent_G(Q) \subseteq H$
again by our hypothesis. Thus, $\Cent_G(D) \subseteq H$.
Therefore, Theorem 5.10.12(ii) of \cite{NagaoTsushima} yields that
$(B')^G = B$. \end{proof}

The following statement helps to determine blocks.

\begin{lem}\label{lem2_4}
Let $N\lhd G$ and $\w\phi\in\IBr(G)$ such that 
$\phi:=\w\phi_N\in\IBr(N)$. 
Let $g \in G$, $\o G:=G/N$ and  $\o g:= gN \in \o G$.
Then the following holds:
\begin{enumerate} \label{lem2_4a}
  \item $\la_{\w\phi}(\Cl_G(g)^+)=  \la_{1_{\o G}}(\Cl_{\o G}(\o g)^+)\, 
\la_{\w\phi_{\langle N,g \rangle }}\left((\Cl_{G}(g)\cap gN)^+\right )$.
  \item \label{lem2_4b} If  $p\nmid |G/N|$, then 
$\la_{1_{\o G}}(\Cl_{\o G}(\o g)^+) \neq 0$,
namely this value is invertible.
\end{enumerate}
\end{lem}

\begin{proof}
Part (a) follows from Lemma 2.5(a) of {\cite{Spaeth_red_BAW}}. 
By definition  $\la_{1_{\o G}}(\Cl_{\o G}(\o g)^+)$ 
and $|\Cl_{\o G}(\o g)|^*$ coincide.
Since $p\nmid |\o G|$ and  $|\Cl_{\o G}(\o g)|$ divides $|\o G|$, 
$|\Cl_{\o G}(\o g)|^*$ is invertible in $k$.
\end{proof}

The following statement seems to be well-known, 
but we give its proof for completeness.

\begin{lem}\label{lem2_5}
Let $N\lhd G$ and $H\leq G$. Let $M:=N\cap H$, $b'\in\Bl(M)$ 
and $c'\in\Bl(H\mid b')$. If $(b')^N$ and $(c')^G$ are defined, then $(c')^G$ 
covers $(b')^N$. 
\end{lem}

\begin{proof}
According to Theorem (9.5) in \cite{Navarro} it is sufficient to show that 
\[ \la_{b'}^N(\Cl_G(n)^+)=\la_{c'}^G(\Cl_G(n)^+) \text{ for every } n\in N.\]
Since $c'$ covers $b'$, Theorem (9.5) in \cite{Navarro} implies 
\[ \la_{b'}(\Cl_H(m)^+)=\la_{c'}(\Cl_H(m)^+) \text{ for every } m\in M.\]
Then the definition of induced blocks 
shows that $ \la_{b'}^N(\Cl_G(n)^+)=0$  
and $\la_{c'}^G(\Cl_G(n)^+)=0$ for every $n\in N$ with 
$\Cl_G(n)\cap H=(\Cl_G(n)\cap N)\cap H= \Cl_G(n)\cap M =\emptyset$. 
For $m\in M$ we have 
\begin{align}
 \la_{b'}^N(\Cl_G(m)^+)=\la_{b'}((\Cl_G(m)\cap M)^+)=
 \la_{c'}((\Cl_G(m)\cap H) ^+)= \la_{c'}^G(\Cl_G(m) ^+).
\end{align}
This proves that $(c')^G$ covers $(b')^N$. 
\end{proof}

The following two technical statements help to control 
blocks that contain certain restrictions or extensions of characters. 

\begin{lem}\label{lem2_6_r_eineRichtung}
Let $N\lhd G$ and $H\leq G$ with $G=NH$. For $M:=N\cap H$ 
let $b'\in\Bl(M)$ be a block that has a defect group $D$ 
with $\Cent_G(D)\subseteq H$. 
Suppose some $\w\phi\in\IBr(G)$ and $\wphip\in \IBr(H)$ satisfy 
$\phi:=\w \phi_N\in\IBr(b)$ and $\phip:=\wphip_M\in\IBr(b')$, 
where $b:=(b')^N$. If $p\nmid |G/N|$ and $\bl(\wphip)^{G}= \bl(\w \phi)$, then
\[\bl(\wphip_{\spann<M,x>})^{\spann<N,x>}= \bl(\w \phi_{\spann<N,x>})
 \text{ for every } x\in H.\]
\end{lem}

\begin{proof}
Note that the considered blocks 
$\bl(\wphip_{\spann<M,x>})^{\spann<N,x>}$ and $\bl(\wphip)^G$ are 
defined since because of $p\nmid |G/N|$ the blocks 
$\bl(\wphip_{\spann<M,x>})$ and $\bl(\wphip)$ have $D$ as a defect group 
and because of $\Cent_G(D)\subseteq H$ the blocks are admissible with respect 
to $\spann<N,x>$ and $G$, respectively.

We prove this by induction on $|\spann<N,x> : N|$. 
We may assume $\bl(\wphip)^G = \bl(\w\phi)$. 

Let $x\in H$. If $|\spann<N,x>:N| = 1$, 
then the assertion is obvious since $\bl(\psi)^N = \bl(\phi)$. 
Now we have to check that 
$\la_{\w\phi_{\spann<N,x>}}(\Cl_{\spann<N,x>}(y)^+)
= \la^{\spann<N,x>}_{\wphip_{\spann<M,x>}}( \Cl_{\spann<N,x>}(y)^+ )$ 
holds for every  $y \in \spann<N,x>$.
Let $y \in \spann<N,x>$.

{\sf Case 1.} Assume first that $\spann<N,y> = \spann<N,x>$. 
This implies that $\Cl_{\spann<N,x>}(y)$ coincides with 
$\Cl_{\spann<N,y>}(y)$ and is an $N$-orbit. Further because of 
$G=NH$ and $N\lhd G$, $H$ permutes the $N$-orbits in $\Cl_{G}(y)$, that is, 
\[\Cl_G(y)  =  \bigcup_{t\in\TT}^.\Big(\Cl_{\spann<N,x>}(y)\Big)^t,
\]
where $\TT$ is a set of $N\Cent_G(y)$-coset representatives in $G$ with 
$\TT\subseteq H$. Note that the union is disjoint. Accordingly,  
$\TT$ is a set of $N\Cent_G(y)\cap H$-coset representatives in $H$. 

Since $|\TT|\gmid |H/M| = |G/N|$, $|\TT|$ is a $p'$-number. By definition, 
\[   \la_{\w\phi_{\spann<N,x>^t}}\Big( \Cl_{\spann<N,x>^t}(y^t)^+ \Big)=  
\la_{\w\phi_{\spann<N,x>}}  \Big( \Cl_{\spann<N,x>}(y)^+ \Big)\]
for any $t\in\mathbb T$. Thus, 
\begin{align}\label{eq_la_wphi}
\la_{\w\phi}(\Cl_G(y)^+) 
= \sum_{t\in\mathbb T}\la_{\w\phi_{\spann<N,x>^t}}
  \Big( \Cl_{\spann<N,x>^t}(y^t)^+ \Big)
= |\TT|^*\,\cdot\,
  \la_{\w\phi_{\spann<N,x>}} \Big( \Cl_{\spann<N,x>}(y)^+ \Big).
\end{align}
Clearly,
\[ 
\Cl_G(y)\cap H = 
 \left(\bigcup_{t\in\TT} \Big( \Cl_{\spann<N,x>}(y)\Big)^t \right) \cap H
= \bigcup_{t\in\TT}\Big(\Cl_{\spann<N,x>}(y)\cap H\Big)^t
\]
since $t\in H$. Thus, by the definition of central functions 
\begin{align}\label{eq_la_wpsi}
\la_{\wphip}\Big( (\Cl_G(y)\cap H)^+ \Big)&=
\sum_{t\in\TT} 
\la_{\wphip_{\spann<M,x>^t}} 
\Big( \left ((\Cl_{\spann<N,x>}(y)\cap H)^t\right )^+\Big)
\\ \nonumber
&= |\TT|^*\,{\cdot}\,
   \la_{\wphip_{\spann<M,x>}} \Big( (\Cl_{\spann<N,x>}(y)\cap H)^+ \Big).
\end{align}
Recall that 
$\la_{\w\phi}(\Cl_G(y)^+) = \la_{\wphip}\Big( (\Cl_G(y)\cap H)^+\Big)$
since $\bl(\wphip)^G = \bl(\w\phi)$.
Moreover, since $|\TT|^* $ is invertible in $k$ and since 
$\Cl_{\spann<N,x>}(y)\cap H = \Cl_{\spann<N,x>}(y)\cap\spann<M,x>$, 
the equations (\ref{eq_la_wphi}) and (\ref{eq_la_wpsi}) imply that
\begin{align*}
\la_{\w\phi_{\spann<N,x>}}(\Cl_{\spann<N,x>}(y)^+)
&= \la_{\wphip_{\spann<M,x>}}\Big( (\Cl_{\spann<N,x>}(y)\cap H)^+ \Big)
\\
&=
\la_{\wphip_{\spann<M,x>}}
\Big( (\Cl_{\spann<N,x>}(y)\cap \spann<M,x>)^+ \Big)
= \la^{\spann<N,x>}_{\wphip_{\spann<M,x>}}( \Cl_{\spann<N,x>}(y)^+ ).
\end{align*}
This proves that $\la_{\w\phi_{\spann<N,x>}}(\Cl_{\spann<N,x>}(y)^+) 
= \la^{\spann<N,x>}_{\wphip_{\spann<M,x>}}( \Cl_{\spann<N,x>}(y)^+ )$ 
whenever $\spann<N,y> = \spann<N,x>$.
          
{\sf Case 2.} Assume now that $\spann<N,y> \lneqq \spann<N,x>$.
Since $|\spann<N,y>:N| < |\spann<N,x>:N|$, it follows from our induction
hypothesis that 
\begin{align} \label{eq_star}
\bl(\wphip_{J})^{\spann<N,y>}  &= \bl(\w\phi_{\spann<N,y>})
\end{align}
for $J:=\spann<N,y>\cap H$. Clearly 
$J\lhd \spann<M,x>$ and $\spann<N,y> \lhd \spann<N,x>$. Accordingly 
$\bl(\wphip_{\spann<M,x>})$ covers $\bl(\wphip_{J})$, and 
$\bl(\w\phi_{\spann<N,x>})$ covers $\bl(\w\phi_{\spann<N,y>})$. 
According to Lemma \ref{lem2_5} we know that 
$\bl(\wphip_{\spann<N,x>\cap H})^{\spann<N,x>}$ covers 
$\bl(\wphip_{J})^{\spann<N,y>}  = \bl(\w\phi_{\spann<N,y>})$. 
For $\Cl_{\spann<N,x>}(y)$ this implies by Passman's result 
in Theorem (9.5) of \cite{Navarro} the following equalities
\begin{align*} 
\la_{\w\phi_{\spann<N,x>}}(\Cl_{\spann<N,x>}(y)^+)
&=\la_{\w\phi_{\spann<N,y>}}(\Cl_{\spann<N,x>}(y)^+)
\\ 
&=\la_{\wphip_{\spann<N,y>\cap H}}^{\spann<N,y>}(\Cl_{\spann<N,x>}(y)^+)
 =\la_{\w\phi'_{\spann<N,x>\cap H}}^{\spann<N,x>}(\Cl_{\spann<N,x>}(y)^+).
\end{align*}
The two cases combined prove $\la^{\spann<N,x>}_{\wphip_{\spann<M,x>}} 
= \la_{\w\phi_{\spann<N,x>}}$, namely,
$\bl(\wphip_{\spann<M,x>})^{\spann<N,x>} = \bl(\w\phi_{\spann<N,x>})$.
This finishes the proof.
\end{proof}

The converse of the above statement is true as well, even without the assumption $p\nmid|G/N|$. 
\begin{lem}\label{lem2_7}
Let $N\lhd G$ and $H\leq G$ with $G=NH$. For $M:=N\cap H$ let 
$b'\in\Bl(M)$ be a block that has a defect group $D$ with $\Cent_G(D)\subseteq H$.
Set $b := (b')^N$.
\begin{enumerate}
\item Suppose some $\w\phi\in\IBr(G)$ and $\wphip\in \IBr(H)$ satisfy 
$\phi:=\w \phi_N\in\IBr(b)$ and $\phip:=\wphip_M\in\IBr(b')$. If
\[\bl(\wphip_{\spann<M,x>})^{\spann<N,x>} 
= \bl(\w \phi_{\spann<N,x>}) \text{ for every } x\in H^0,\]
then $\bl(\wphip)^{G}= \bl(\w \phi)$.
\item  Suppose some $\w\chi\in\Irr(G)$ and $\wchip\in \Irr(H)$ satisfy 
$\chi:=\w \chi_N\in\Irr(b)$ and $\chip:=\wchip_M\in\Irr(b')$. If
\[\bl(\wchip_{\spann<M,x>})^{\spann<N,x>} 
= \bl(\w \chi_{\spann<N,x>}) \text{ for every } x\in H^0,\]
then $\bl(\wchip)^{G}= \bl(\w \chi)$.
\end{enumerate}
\end{lem}

\begin{proof}
We restrict ourselves to prove part (a), since  the proof of part (b) 
works analogously. 

It is sufficient to verify that $\la_{\wphip}((\Cl_G(g)\cap H)^+)
=\la_{\w\phi}(\Cl_G(g)^+)$ for every $g\in G$. 
Note that according to \cite[Theorem 3.6.24(i)]{NagaoTsushima} it 
is sufficient to prove the equality on $p$-regular classes.
Let $g\in G^0$.

{\sf Case 1.} Assume first that $\Cl_G(g)\cap H\neq\emptyset$. 
Let $x\in \Cl_G(g)\cap H$, and hence $\Cl_G(g) = \Cl_G(x)$. 
Let $\TT\subseteq H$ be a set of $N\Cent_G(x)$-coset representatives in $G$. 
The proof of {\sf Case 1} in Lemma \ref{lem2_6_r_eineRichtung} shows that 
\begin{align*} 
\la_{\w\phi}(\Cl_G(x)^+) &= |\TT|^* \,{\cdot}\,
\la_{\w\phi_{\spann<N,x>}} \Big( \Cl_{\spann<N,x>}(x)^+ \Big)
\qquad \text{ and }
\\ 
\la_{\wphip}\Big( (\Cl_G(x)\cap H)^+ \Big)
&= |\TT|^*\,{\cdot}\,
   \la_{\wphip_{\spann<M,x>}} \Big( (\Cl_{\spann<N,x>}(x)\cap H)^+ \Big).
\end{align*}
(Note that for those formulas the assumption $p\nmid |G/N|$ is not used.)
Since  $\Cl_{\spann<N,x>}(x) \cap \spann<M,x> = \Cl_{\spann<N,x>}(x) \cap H$, 
the assumption $\la^{\spann<N,x>}_{\wphip_{\spann<M,x>}} 
= \la_{\w\phi_{\spann<N,x>}}$ implies that
\begin{align*}
\la_{\w\phi_{\spann<N,x>}}(\Cl_{\spann<N,x>}(x)^+)
&= \la_{\wphip_{\spann<M,x>}}
   \Big( (\Cl_{\spann<N,x>}(x)\cap H)^+\Big).
\end{align*}
Thus, we get $\la_{\w\phi}(\Cl_G(g)^+) = \la^G_{\wphip}(\Cl_G(g)^+)$.

{\sf Case 2.} Suppose that $\Cl_G(g) \cap H = \emptyset$. By Corollary 5.10.13 of 
\cite{NagaoTsushima}, $\bl(\w\phi) = c^G$ for some $c\in\Bl(H)$. 
Then, $\la_{\w\phi}(\Cl_G(g)^+) = \la^G_c(\Cl_G(g)^+) = \la_c((\Cl_G(g)\cap H)^+) = 0$, 
and $\la^G_{\wphip}(\Cl_G(g)^+) = \la_{\wphip}( (\Cl_G(g)\cap H)^+) = 0$.

Together with the first case this shows  $\bl(\wphip)^G = \bl(\w\phi)$.
\end{proof}

\section{Extending to $G[b]$}\label{sec3}
We study here the group $G[b]$ introduced 
by Dade in (2.9) \cite{DadeBlockExtensions}, 
that is a normal subgroup of $G_b$ according to Proposition 2.17 of \cite{DadeBlockExtensions}. 
This group plays a central role in the study 
of the ramification of blocks, covering a 
given block of a normal subgroup. 
In \cite{Murai_Dade}, Murai gives 
some character theoretic interpretation of this group.
We use the notation from \cite{Murai_Dade}. 

\begin{notation}\label{Def_G[b]}
Let $N\lhd G$ and $b\in\Bl(N)$. Then let
\[ G[b]:=\{ g\in G_b\mid 
\la_{b^{(g)}}(\Cl_{\spann<N,g>}(y)^+)\neq 0   
\text{ for some } y\in gN\},\]
where for every $g\in G$ we denote  
by $b^{(g)}$ an arbitrary block in $\Bl(\spann<N,g>\mid b)$. 
Note that by this definition, $G[b]$ is well-defined and 
$G[b]$ does not depend on the actual choice of the blocks $b^{(g)}$,
see the next proposition.
Let $e$ be the central-primitive idempotent of $\mathcal O N$ associated to
$b$. 
This group controls the ramification of blocks 
above the block of a normal subgroup. 
By definition $G[b]\lhd G_b$. 
\end{notation}

\begin{prop}\label{prop_def_alt}
The group $G[b]/N$ satisfies 
\[G[b]=\{ g \in G\mid  (e \mathcal O N g) (e \mathcal O N g^{-1})= e \mathcal O N \}.\]
Accordingly $G[b]/N$ coincides with the group introduced in (2.9a) 
of \cite{DadeBlockExtensions}, see also page 210 of 
\cite{Kuelshammer} and Lemma 3.2 of \cite{HidaKoshitani}. 
\end{prop}

\begin{proof} 
The characterization of $G[b]$ given by K\"ulshammer in 
\cite[p.210]{Kuelshammer} shows that $G_0[b]\leq G[b]$ 
whenever $N\leq G_0 \leq G$. 
Hence any element $x\in G$ is contained 
in $G[b]$ if and only if for $G_0:=\spann<N,x>$ the element satisfies  
$x\in G_0[b]$. In this case $G_0/N$ is cyclic and by Corollary 3.3 
of \cite{Murai_Dade} the group $G_0[b]$ can be defined as above.
\end{proof}
For a proof we recall several basic known results on the group $G[b]$. 

\begin{lem}\label{KueWatanabe}
Assume that $N \lhd G$ and $b\in\bl(N)$.
\begin{enumerate}
    \item \label{KueWatanabe_a}
	Every $\chi\in\Irr(b)$ and every
$\phi\in\IBr(b)$ is $G[b]$-invariant.
    \item \label{KueWatanabe_b}
    If $G = G[b]$ and $J$ is any subgroup of $G$ with
$N\leq J\leq G$, then $J = J[b]$.
    \item \label{KueWatanabe_c}
	If $G = G[b]$, $M\leq N$ and $b'\in \Bl(M)$ with $(b')^N=b$, 
then $G = N\Cent_G(D) = N\NNN_G(D)$ for every defect group $D$ of $b'$.
\end{enumerate}
\end{lem}

\begin{proof} Part (a) is given by  (13.6) and Proposition 13.3  of \cite{DadeBlockExtensions}.
Part (b) follows from Dade's definition of 
$G[b]$ given in (2.9a) of \cite{DadeBlockExtensions}, see also Proposition \ref{prop_def_alt}. For (c) 
we note that there is a defect group $D_0$ of $b$ with 
$D\subseteq D_0$, see Theorem (4.14) of \cite{Navarro}.
Then, the assumption $G = G[b]$ implies that
$G = N\Cent_G(D_0)$ 
by  Corollary 12.6 of \cite{DadeBlockExtensions}, 
see also Theorem 3.13 and Remark 3.14 of \cite{Murai_Dade}. Hence 
$G = N\Cent_G(D_0) \subseteq N\Cent_G(D) 
\subseteq N\NNN_G(D)$.
\end{proof}

This group $G[b]$ has the following behavior when going to quotients of $G$. 

\begin{lem}\label{lem3_5}
Let $N\lhd G$, $Z\leq \Z(G)$ with $N\cap Z=1$
and $b \in\Bl(N)$. Let $\o G:=G/Z$. Then $\o G[\o b]=G[b]/Z$ 
where $\o b\in \Bl((N \times Z)/Z)$ is the block associated with $b$ via the canonical isomorphism 
between $N$ and $(N\times Z )/Z$. 
\end{lem}

\begin{proof}
Let $x\in G$, $\chi\in\Irr(b)$ and $\o\chi \in \Irr(NZ/Z)$, 
the character associated to $\chi$.
If $\chi$ is not $x$-invariant or equivalently 
$\o\chi$ is not $xZ$-invariant then  $x\notin G[b]$ 
and $xZ\notin \o G[\o b]$ because of Lemma \ref{KueWatanabe}(a).
 
Otherwise there exists an extension $\w \chi\in\Irr(\spann<N,Z,x>)$ 
 of $\chi$ with $Z\leq \ker(\w\chi)$. 
(Here $\w\chi$ exists since 
 $\spann<N,Z,x>/\spann<N,Z>$ is cyclic.) 

Let $y\in Nx$. Since $Z\leq \Z(G)$, the sets 
$\Cl_{\spann<N,x>}(y)$ and $\Cl_{\spann<N,Z,x>}(y)$ coincide. 
This proves
\[\la_{\w\chi}(\Cl_{\spann<N,Z,x>}(y)^+)
= \la_{\w\chi_{\spann<N,x>}}(\Cl_{\spann<N,x>}(y)^+).\]
Further $\w\chi$ defines $\kappa\in \Irr(\spann<N,Z,x>/Z)$ 
and $\bl(\kappa)$ covers $\o b$. By definition we see that 
$\la_{\kappa}(\Cl_{\spann<N,Z,x>/Z}(yZ)^+)
= \la_{\w\chi}(\Cl_{\spann<N,Z,x>}(y)^+)$.

If $x \in G[b]$, then there exists some $y\in Nx$ 
such that $\la_{\w\chi_{\spann<N,x>}}(\Cl_{\spann<N,x>}(y)^+)\neq 0$, 
and hence  $xZ\in \o G[\o b]$. Analogously if $x \in \o G[\o b]$, 
then there exists some $yZ\in NZx/Z$ such that
$\la_{\kappa}(\Cl_{\spann<N,Z,x>/Z}(yZ)^+)\neq 0$, 
and hence  $\la_{\w\chi_{\spann<N,x>}}(\Cl_{\spann<N,x>}(yZ)^+)\neq 0$. 
This implies that $\o G[\o b]=G[b]/Z$.
\end{proof}

We recall the following fact about characters belonging to isomorphic blocks. 

\begin{lem}\label{IsomorphicBlock}
Assume that $N\lhd G$ with $p \nmid |G/N|$, $b\in\Bl(N)$,   
$B\in\Bl(G\mid b)$ and $G = G[b]$. In addition, if there is 
$\wphi \in \IBr(B)$ with $\wphi_N\in\IBr(b)$, or $\wchi \in \Irr (B)$ 
with $\wchi_N\in\Irr(b)$, 
then $B$ and $b$ are isomorphic via restriction, 
and hence $b \cong B{\downarrow}_{N\times N}$.
\end{lem}

\begin{proof}
We know by Fong's result Theorem 5.5.16(ii) of \cite{NagaoTsushima} 
that $b$ and $B$ have a common defect group since $p \nmid |G/N|$. 
Hence, it follows from Theorem 7 of \cite{Kuelshammer} 
(see Theorem 3.5 of \cite{HidaKoshitani}) that $B$ and $b$ 
are naturally Morita equivalent of degree $n$ for some integer $n \geq 1$. 
Assume that the first case occurs, namely, that there is 
$\wphi \in \IBr (B)$ with $\wphi_N \in \IBr(b)$. Then
it holds $n = 1$. Therefore we get the assertion 
by Theorem 4.1(7) of \cite{HidaKoshitani}. 
For the second case, a similar argument works 
by Proposition 2.4 of \cite{HidaKoshitani}.
\end{proof}  

Under additional conditions, extensions of characters can 
be uniquely determined by a block.

\begin{lem}\label{lem3_7}
Let $N\lhd G$, $b\in\Bl(N)$, and $B\in\Bl( G\mid b)$. 
Assume that $G/N$ is cyclic and $G[b]= G$.
For every  $\phi\in\IBr(b)$ there exists a unique extension 
$\widetilde \phi\in\IBr(B)$ of $\phi$.
\end{lem}

\begin{proof}
Since $G = G[b]$, $b$ is $G$-invariant. Hence $\phi$ is $G$-invariant
by Lemma \ref{KueWatanabe}(a).
Since $B$ covers $b$,
there exists some $\w\phi\in\IBr(G\mid \phi)\cap \IBr(B)$ 
and this is an extension of $\phi$ since $G/N$ is cyclic,
namely $\w\phi_N = \phi$, 
see Corollary (8.20) and Theorem (9.2) of \cite{Navarro}.

Next, we show the uniqueness of $\w\phi$.
Let $\w\phi' \in \IBr(B)$
such that $(\w\phi')_N = \phi$. By the definition of $G[b]$
there is an element $y \in G$ with $\langle N,y\rangle = G$
and $\lambda_B(\Cl_G(y)^+) \not= 0$. Since $G/N = \langle yN\rangle$
is cyclic, there exists some $y' \in G$ such that 
$\langle y'N\rangle$ is the Hall $p'$-subgroup of $G/N$. 
Then we can write $y'N = y^iN$ for
some $i\in\mathbb Z$. Thus,
$0 \not= \Big( \lambda_B(\Cl_G(y)^+) \Big)^i
       = \lambda_B\Big( (\Cl_G(y)^+ )^i \Big)$.
It is easily seen that
$(\Cl_G(y)^+ )^i \in \sum_{z\in y^iN}\mathbb Z \Cl_G(z)^+$,
see the proof of Corollary 3.3 of \cite{Murai_Dade}.
Hence we can write 
$(\Cl_G(y)^+)^i = \sum_{z\in y^iN}a_z{\cdot}(\Cl_G(z)^+)$
for some integers $a_z$, so that
$0 \not= \sum_{z\in y^iN}a_z^*{\cdot}\lambda_B(\Cl_G(z)^+)$.
Thus there exists an element $z\in y^iN$ such that
$\lambda_B(\Cl_G(z)^+) \not= 0$. Obviously,  $zN = y^iN = y'N$.

Now, since $\w\phi$ and $\w\phi'$ are both extensions of $\phi$,
there is a linear character $\zeta\in\IBr(G/N)$ with $\w\phi' = \w\phi\zeta$,
see Corollary (8.20) of \cite{Navarro}.
Clearly $\zeta$ is linear since $G/N$ is cyclic.
Hence,
\begin{align*}
0 &\not= \ \lambda_B(\Cl_G(z)^+)
= \lambda_{\w\phi'}(\Cl_G(z)^+)
= \lambda_{\w\phi\zeta}(\Cl_G(z)^+)=
\\
&= \lambda_{\w\phi}(\Cl_G(z)^+){\cdot}\zeta(z)^*
 = \lambda_B(\Cl_G(z)^+){\cdot}\zeta(z)^*.
\end{align*}
Hence $ \lambda_B(\Cl_G(z)^+) \not= 0$ and $\zeta(z)^* = 1 \in k$. Thus,
$$ 1 = \zeta(z)^* = \zeta(zN)^* = \zeta(y^iN)^* = \zeta(y'N)^*.$$
Hence $\zeta = 1_{G/N}$ since $\langle y'N\rangle$ is a 
Hall $p'$-subgroup of $G/N$. This implies $\w\phi' = \w\phi$.
\end{proof}

\begin{lem}\label{prop_cor3_3}
Let $N \lhd G$, $H\leq G$ and $M:=N\cap H$. 
Let $b'\in \Bl(M)$ be a block that has a defect group 
$D$ with $\Cent_G(D)\subseteq H$. 
If $G[b]=G$ for $b:=(b')^N\in \Bl(N)$, then $H= H[b']$.
\end{lem}
Note that this statement follows from Corollary 12.6 of 
\cite{DadeBlockExtensions} in the case where $b$ and 
$b'$ have a common defect group. 
\begin{proof}
Let $h\in H$. Since $h\in G[b]$, there exists a block 
$B\in\Bl(\spann<N,h> | b )$ and $y\in hN$ such that 
$\la_B(\Cl_{\spann<N,h>}(y)^+)\neq 0$.
Clearly, $\spann<N,h>/N$ is cyclic,
and $b$ is $\spann<N,h>$-invariant since $h \in G[b]$.
Moreover, according to the above and Corollary 3.3 of \cite{Murai_Dade}, 
it holds $\spann<N,h>[b] = \spann<N,h>$.
Hence, by Proposition 3.4(ii) of \cite{HidaKoshitani}, 
$b\mid (B{\downarrow}_{N\times N})$.
Hence, by Lemma \ref{lem2_3}(b) there exists a block 
$B'\in\Bl({\spann<M,h>}\mid b')$ with $(B')^{\spann<N,h>}=B$.
The block $B'$ satisfies 
\[ \la_{B'}\Big((\Cl_{\spann<N,h>}(y)\cap \spann<M,h>)^+\Big)
  =\la_B(\Cl_{\spann<N,h>}(y)^+)\neq 0 .\]
This proves that there exists an element 
$z\in hM$ with $\la_{B'}(\Cl_{\spann<M,h>}(z)^+)\neq 0$. 
This implies that $h\in H[b']$. Hence $H =H[b']$.
\end{proof}

The following lemma shows 
how a character correspondence can be deduced
from existence of a character of $G$ 
with several properties. 

\begin{lem}\label{prop3_9}
Let $N \lhd G$, $H \leq G$, and $M:=N\cap H$. 
Let $b'\in \Bl(M)$ be a block that has a defect group $D$ 
with $\Cent_G(D)\subseteq H$. Assume that $G = G[b]$ for 
$b:=(b')^N$, and $p\nmid |G/N|$. Suppose that there is some 
$\wtheta\in\IBr(G)\cup\Irr(G)$ such that 
$\theta:=\wtheta_N\in\IBr(b)\cup \Irr(b)$. 
Then there is the unique block $B'\in\Bl(H\mid b')$ with 
$(B')^G=\bl(\wtheta)$, and such that
\enumalph
\begin{enumerate}
	\item the map $\wphip\mapsto\wphip_M$ 
defines a bijection $\IBr(B') \rightarrow \IBr(b')$.
	 \item the map $\w\chi'\mapsto\w\chi'_M$ 
defines a bijection $\Irr(B') \rightarrow \Irr(b')$.
\end{enumerate}
\end{lem}

\begin{proof} Let $B := \bl(\w\theta)$ 
and $b:= \bl(\theta)$. By Lemma \ref{IsomorphicBlock}, 
$B{\downarrow}_{N\times N} \cong b$, via restriction, 
see Definition 4.3 of \cite{HidaKoshitani}. Now, 
by Lemma \ref{lem2_3}(b), there exists a block $B' \in \Bl(H\mid b')$ 
such that $(B')^G = B  $
and hence $B'\mid (B{\downarrow}_{H\times H})$, 
see Theorem 5.10.11 of \cite{NagaoTsushima}.

On the other hand, by Lemma \ref{prop_cor3_3}, 
$H = H_{b'} = H[b']$. Since $B'$ covers $b'$, 
it follows from Proposition 3.4(ii) of \cite{HidaKoshitani} that
$ B'{\downarrow}_{M \times M}$ is the $n$-times direct product
of $b'$, denoted by $  n \times  b'$ 
for an integer $n\geq 1$. Thus,
\begin{align*}
   b{\downarrow}_{M\times M} 
&\cong \ (B{\downarrow}_{N\times N}){\downarrow}_{M\times M} \ 
= \ B{\downarrow}_{M\times M}
\ = \ (B{\downarrow}_{H\times H}){\downarrow}_{M\times M} \ 
\\
&= \ \left (B'\oplus (k[H\times H]\text{-mod})\right ){\downarrow}_{M\times M}
\ \cong \ 
(n \times b') \oplus (k[M\times M]\text{-mod}),
\end{align*}
so that $b'$ is an indecomposable direct summand 
of $b{\downarrow}_{M\times M}$ 
with multiplicity at least $n$. 
Further since $b$ is an indecomposable direct summand of 
$kN$, we see that $b'$ is an indecomposable direct summand of 
$(kN){\downarrow}_{M\times M}$ with multiplicity at least $n$. 
Note that $\Cent_N(D) \subseteq M$. 
By Theorem 5.10.12(i) of \cite{NagaoTsushima}, 
$b'$ is an indecomposable direct summand of 
$(kN){\downarrow}_{M\times M}$ with multiplicity one. 
This proves that $n = 1$, that is, 
$B'{\downarrow}_{M\times M} \cong b'$. 
Thus, again by Theorem 4.1 of \cite{HidaKoshitani}, 
there exists a bijection
\begin{align*}
\IBr(B') \longrightarrow \IBr(b')  &
\text{ with }  \kappa \longmapsto \kappa_M .
\end{align*}
Hence there exists a unique 
$\wphip \in \IBr(B') \, \cap \, \IBr(H\mid \phip)$,
and it satisfies $(\wphip)_N = \phi'$. Similar considerations apply to 
ordinary characters.

Next we show that $B'$ is the unique block 
in $\Bl(H\mid b')$ with $(B')^G = B$. 
Let $B_0'\in \Bl(H\mid b')$ with $(B_0')^G = B$ and $B_0' \not= B'$. 
The considerations above show that 
${B'_0}{\downarrow}_{M\times M}  \cong b'$, 
via restriction.  Then, since $(B_0')^G = B = (B')^G$,
we have $({B_0'}\oplus B') \mid (B{\downarrow}_{H\times H})$, 
see Theorem 5.10.11 of \cite{NagaoTsushima}. Thus,
\begin{align*}
b\downarrow_{M\times M} 
&= (B\downarrow_{N\times N})\downarrow_{M\times M} 
= B\downarrow_{M\times M}
= (B\downarrow_{H\times H})\downarrow_{M\times M} 
\\
&= [B_0' \oplus B' \oplus (k[H\times H]
\text{-mod})]{\downarrow}_{M\times M} = b' 
\oplus b' \oplus (k[M\times M]\text{-mod}),
\end{align*}
so that 
\begin{align*}
 (b'\oplus b') \mid  (b{\downarrow}_{M\times M}) 
\mid (kN){\downarrow}_{M\times M}   & =  b' \oplus (\text{the others}),
 \end{align*}
a contradiction. 
\end{proof}

An analogous result is the following. 

\begin{lem}\label{prop3_10}
Let $N \lhd G$, $H\leq G$, and $M:=N\cap H$. 
Let $b'\in \Bl(M)$ be a block that has a defect group $D$ 
with $\Cent_G(D)\subseteq H$. Assume that $G = G[b]$ for 
$b:=(b')^N$,  and $p\nmid |G/N|$. Suppose that 
there is some $\wthetap\in\IBr(H) \, \cup \, \Irr(H)$ such that 
$\thetap:=\wthetap_M$ is irreducible and belongs to $b'$. 
Then \enumalph
\begin{enumerate}
	\item the map $\w\phi\mapsto\w\phi_N$ defines a bijection
		 $\IBr(\bl(\wthetap)^{ G}) \rightarrow \IBr(b)$.
	 \item the map $\w\chi\mapsto\w\chi_N$ defines a bijection 
	 $\Irr(\bl(\wthetap)^{ G}) \rightarrow \Irr(b)$.
\end{enumerate}
\end{lem}

\begin{proof}
Set $B' := \bl(\w\theta') \in\Bl(H)$. Clearly, $B'$ covers $b'$.
Furthermore $B := (B')^G$ is defined and 
by Lemma \ref{lem2_5}, $B$ covers $b$.
Thus, Lemma \ref{IsomorphicBlock} applies, so the assertion
follows from \cite[Theorem 4.1(5) and (12)]{HidaKoshitani}.
\end{proof}

Finally we prove Theorem \ref{thm1_3},
from which we later deduce Theorems \ref{thmA} and \ref{thm_gen_HK}. 
If $N\lhd G$ and $b\in\Bl(N)$ with $G[b]=G$, 
we can predict whether a character of 
$b$ has an extension to $G$ lying in a specific block by considering characters of local subgroups. 
In the situation where the defect groups of $b$ and $b'$ coincide 
Theorem 12.3 and Corollary 12.6 of \cite{DadeBlockExtensions} 
proves that the extensibility question is determined locally. 

\renewcommand{\proofname}{\bf Proof of Theorem \ref{thm1_3}(a)(1) and (b)(1)}
\begin{proof}
We first prove that there exists a character 
$\chi\in\Irr(b)$ of height zero that extends to 
some $J_0$ with $N\leq J_0\leq G$ and $J_0/N\in\Syl_p(G/N)$. 
Afterwards (a)(1) is a consequence of (b)(1), 
that we verify below. Notice that $G=NH$, according to Lemma 
\ref{KueWatanabe}(c). 

Some defect group $\w D$ of $\bl(\wchip)^G$ satisfies 
$D\leq \w D$, see Theorem (9.24) of \cite{Navarro}. 
Since $\wchip$ is an extension of $\chip$, the group $DM/M$ 
is a Sylow $p$-subgroup of 
$H/M$, see Proposition 2.5(d) of \cite{NavarroSpaeth}. 
Since $G=NH$, this implies that 
$DN/N= \w D N/N$ is a Sylow $p$-subgroup of $G/N$.

Let $c\in\Bl(N\w D)$ be the unique block covering $b$, 
see Corollary (9.6) of \cite{Navarro}. 
According to Problem (9.4) of \cite{Navarro}, 
$\w D$ is a defect group of $c$. 
Let $\mu\in\Irr(c)$ be a  height zero character 
and $\chi\in\Irr(N\mid \mu)$. Then by Corollary (9.18) 
of \cite{Navarro} the multiplicity of $\chi$ in $\mu_N$ 
is not divisible by $p$, and $\chi$ is invariant 
in $\w DN$. Together with Exercise (6.7) of \cite{Isa} 
this proves that $\mu$ is an extension of $\chi$. 
Hence there exists a character $\chi\in\Irr(b)$ 
with height zero that extends to $\w DN$.

In the following let $\chi\in\Irr(b)$ be a character 
that extends to $\w DN$. 
Let $J$ be any subgroup with $N\leq J\leq G$ 
and $p\nmid |J/N|$. According to Lemma \ref{prop3_10} 
there exists a unique character $\kappa^{(J)}\in\Irr(J)$, 
an extension of $\chi$ such that 
$\bl(\kappa^{(J)})=\bl(\wchip_{H\cap J})^J$. 
Note that by their definition for every $x\in H$ the characters 
$\kappa^{(J)}$ and $\kappa^{(J^x)}$ satisfy
\begin{align}\label{eq_kappa_def}
\kappa^{(J)} (j)= \kappa^{(J^x)} (j^x) \text{ for every } j 
\in J.\end{align}
Furthermore according to Lemma \ref{lem2_6_r_eineRichtung}, 
$\kappa^{(J)}$ satisfies 
$\bl((\kappa^{(J)})_{\spann<N,j>})=
\bl(\wchip_{H\cap \spann<N,j>})^{\spann<N,j>} $ 
for every $j\in J$. By the uniqueness in 
Lemma \ref{prop3_10} this implies 
\begin{align}\label{eq_kappa}
(\kappa^{(J)})_{\spann<N,j>}
= \kappa^{(\spann<N,j>)}\text{ for every }j\in J.
\end{align}
This implies that $\chi$ extends to $G$ 
since for every prime $\ell$ the character 
$\chi$ extends to some group $K$ with 
$K/N\in\Syl_{\ell}(G/N)$, see Corollary (11.31) of \cite{Isa}. 
Let $\theta\in\Irr(G)$ 
be an extension of $\chi$ to $G$. 

For every $g\in G^0$ we define a unique character 
$\epsilon^{(g)}\in\Irr(\spann<N,g>)$ by
\[ \kappa^{(\spann<N,g>)}= \theta_{\spann<N,g>} \epsilon^{(g)} .\]
By this definition $\epsilon^{(g)}$ is a character of 
$G/N$. Now let $\epsilon: G \rightarrow \CC$ be 
the map with $ \epsilon(g):=\epsilon^{(g_{p'})}(g_{p'})$ 
for every $g\in G$, where $g_{p'}$ is the $p'$-part of $g$. 

In the following we verify several properties of $\epsilon$ 
in order to apply Brauer's characterization of characters later
and prove thereby that $\epsilon$ is a linear character of $G$. 
Note that by definition $\epsilon$ is constant on $N$-cosets, 
and  by definition $\epsilon(g)$ is a root of unity for every 
$g\in G$, since $\epsilon^{(g)}$ is a linear character. Hence 
$\epsilon$ has norm $1$.

Now we compare $\epsilon(g)$ and $\epsilon(g^x)$ for 
$g,x\in G$. Since $\epsilon$ is constant on $N$-cosets, we 
may assume that $g, x \in H$. Now $\epsilon^{(g)}$ and 
$\epsilon^{(g^x)}$ are uniquely defined by equations 
\[ 
\kappa^{(\spann<N,g>)}
= \theta_{\spann<N,g>} \epsilon^{(g)} \text{ and }
\kappa^{(\spann<N,g^x>)}= \theta_{\spann<N,g^x>} \epsilon^{(g^x)}.
\]
Since $\theta\in\Irr(G)$, we see from (\ref{eq_kappa_def}) that 
$\epsilon^{(g^x)}(g^x)=\epsilon^{(g)}(g)$. Accordingly 
$\epsilon$ is a class function of $G$. 

Let $E$ be an elementary subgroup of $G$, hence it is 
a direct product of a $p$-group and a $p'$-group 
$E_{p'}$. By the definition of $\epsilon$ the map 
$\epsilon_E$ is a character if and only of $\epsilon_{E_{p'}}$ 
is a character. Hence we may assume that $E$ is a $p'$-group. 
For every $e\in E$, $\epsilon^{(e)}\in \Irr(\spann<N,e>)$ is defined by
\[ \kappa^{(\spann<N,e>)}=\theta_{\spann<N,e>} \epsilon^{(e)}.\]
Because of (\ref{eq_kappa})  $\epsilon^{(e)}$ satisfies also 
\[ (\kappa^{(\spann<N,E>)})_{\spann<N,e>}
=\theta_{\spann<N,e>} \epsilon^{(e)}.\]
Accordingly $\epsilon_E$ coincides with  
a linear character  $\zeta\in\Irr(\spann<N,E>)$ 
given by $\kappa^{(\spann<N,E>)}= \theta_{\spann<N,E>}\zeta$ . 

Applying now Brauer's characterization of characters 
known from Corollary (8.12) of \cite{Isa} we see that 
$\epsilon$ is a linear character of $G$. Accordingly 
$\w\chi:=\theta\epsilon$ is a well-defined character. 
Since $\epsilon(1)=1$ by definition, $\w\chi$ is an 
extension of $\chi$, and $\w \chi$ satisfies 
\[ \bl(\w \chi_{\spann<N,g>})
=\bl(\wchip_{\spann<N,g>\cap H})^{\spann<N,g>} \text{ for every } g\in G.\]
According to Lemma \ref{lem2_7} this implies 
\begin{align}
 \bl(\w \chi_{J})&=\bl(\wchip_{J\cap H}) 
\text{ for every } J \text{ with } N\leq J\leq G.
\qedhere
\end{align}
\end{proof}

\renewcommand{\proofname}{\bf Proof of Theorem \ref{thm1_3}(c)(1)}
\begin{proof}
Notice that $G=NH$ according to Lemma \ref{KueWatanabe}(c), 
and $H[b']=H$ by Lemma \ref{prop_cor3_3}.

Let $\phi\in\IBr(b)$. According to Lemma \ref{KueWatanabe}(a), 
$\phip$ is $H$-invariant. Let $J$ be an elementary subgroup with 
$N\leq J\leq G$ and $p\nmid |J/N|$. According to Lemma \ref{prop3_10} 
there exists a unique character $\kappa^{(J)}\in\IBr(J)$, 
extension of $\phi$ such that $\bl(\kappa^{(J)})=\bl(\wphip_{H\cap J})^J$. 
Note that by the definition,
for every $x\in H$ the characters $\kappa^{(J)}$ and 
$\kappa^{(J^x)}$ satisfy  
\begin{align}\label{eq_kappa_def_IBr}
\kappa^{(J)} (j)= \kappa^{(J^x)} (j^x) \text{ for every } j \in J.
\end{align}
Furthermore according to Lemma \ref{lem2_6_r_eineRichtung}, 
$\kappa^{(J)}$ satisfies 
$\bl((\kappa^{(J)})_{\spann<N,j>})
=\bl(\wphip_{H\cap \spann<N,j>})^{\spann<N,j>} $ for every $j\in J$. 
By the uniqueness in Lemma \ref{prop3_10} this implies 
\begin{align}\label{eq_kappa_IBr}
(\kappa^{(J)})_{\spann<N,j>}= \kappa^{(\spann<N,j>)}\text{ for every }j\in J.
\end{align}
According to Theorem (8.29) of \cite{Navarro} 
this implies that $\phi$ extends to $G$, since for every 
prime $\ell\neq p$ the character $\phi$ extends to some group 
$K$ with $N\leq K\leq G$ and $K/N\in\Syl_{\ell}(G/N)$. 
Let $\theta\in\IBr(G)$ be an extension of $\phi$ to $G$. 

For an ordinary character $\mu$ we denote by $\mu^0$ its 
restriction to the $p$-regular elements. For every $g\in G^0$ 
we let $\epsilon^{(g)}\in\Irr(\spann<N,g>)$ be the unique 
character with $N\leq \ker(\epsilon^{(g)})$ and 
\[ \kappa^{(\spann<N,g>)}= \theta_{\spann<N,g>} (\epsilon^{(g)})^0 .\]
Note that $\spann<N,g>/N$ is a $p'$-group, and hence every 
$\mu\in\IBr(\spann<N,g>/N)$ coincides with $\nu^0$ for  
a unique character $\nu\in\Irr(\spann<N,g>/N)$. By the 
definition, $\epsilon^{(g)}$ is uniquely defined. Again 
let $\epsilon: G \rightarrow \CC$ by 
$\epsilon(g):=\epsilon^{(g_{p'})}(g_{p'})$ for every $g \in G$,
where $g_{p'}$ is the $p'$-part of $g$. 

Like before we now verify that $\epsilon\in\Irr(G)$. 
The definition implies again that $\epsilon$ has norm $1$.

The considerations above prove again that $\epsilon$ 
is a class function of $G$, namely, $\epsilon(g)=\epsilon(g^x)$ 
for every $g,x\in G$. 
Since $\epsilon$ can be seen as a character of $G/N$,
we may assume that $g,x \in H$. Now $\epsilon^{(g)}$ and 
$\epsilon^{(g^x)}$ are uniquely defined by the following equations 

\[ \kappa^{(\spann<N,g>)}
= \theta_{\spann<N,g>} (\epsilon^{(g)})^0 \text{ and }
\kappa^{(\spann<N,g^x>)}= \theta_{\spann<N,g^x>} (\epsilon^{(g^x)})^0.
\]
Recall that $\theta\in\IBr(G)$. Because of (\ref{eq_kappa_def_IBr}) 
we see that $\epsilon^{(g^x)}(g^x)=(\epsilon^{(g)})(g)$. 

Next we want to show that $\epsilon_E$ is a character for every elementary 
subgroup $E\leq G$. Like before we may assume by the definition of 
$\epsilon$ that $p\nmid |E|$. 
For every $e\in E$, $\epsilon^{(e)}\in \Irr(\spann<N,E>)$ 
is defined by 
$\kappa^{(\spann<N,e>)}=\theta_{\spann<N,e>} (\epsilon^{(e)})^0$, 
and because of (\ref{eq_kappa_IBr})  $\epsilon^{(e)}$ satisfies also 
$ (\kappa^{(\spann<N,E>)})_{\spann<N,e>}
=\theta_{\spann<N,e>} (\epsilon^{(e)})^0$. Accordingly 
$\epsilon_E$ is a linear character $\zeta\in\Irr(\spann<N,E>/N)$ 
that satisfies  $\kappa^{(\spann<N,E>)}= \theta_{\spann<N,E>}\zeta^0$. 

By \cite[(8.12)]{Isa} we see that $\epsilon\in\Irr(G)$. Hence 
$\w\phi:=\theta\epsilon^0\in\IBr(G)$ is a well-defined extension 
of $\phi$ with 
\[ \bl(\w \phi_{\spann<N,g>})
=\bl(\wphip_{\spann<N,g>\cap H})^{\spann<N,g>} 
\text{ for every } g\in G^0.\]
According to Lemma \ref{lem2_7} this implies 
\begin{align}
 \bl(\w \phi_{J})&=\bl(\wphip_{J\cap H}) 
\text{ for every } J \text{ with } N\leq J\leq G. \qedhere
\end{align}
\end{proof}

Slightly altering the proof we can prove Theorem 
\ref{thm1_3}(a)(2), (b)(2) and (c)(2).

\renewcommand{\proofname}
{\bf Proof of Theorem \ref{thm1_3}(a)(2), (b)(2) and (c)(2)}
\begin{proof}
We only sketch here a proof since it is very similar to 
the proofs given before.
The block $b'$ is invariant in $H$ since
$H[b']=H$ by Lemma \ref{prop_cor3_3}. 
Like in the proof of Theorem \ref{thm1_3}(a)(1) and (b)(1),  
we see that a character in $\Irr(b')$ with height zero extends 
to $J_0\cap H$. 
The remaining considerations of the proofs 
of Theorem \ref{thm1_3}(a)(1), (b)(1) and (c)(1), can be applied 
with only small changes and prove the statement. 
\end{proof}

\section{Maps between characters and blocks}
The extensions constructed in Theorem 
\ref{thm1_3}  
give rise to several character correspondences. 
We restrict ourselves to giving here one example. 
Besides the existence of the extensions, 
the equations $(*)$ and $(**)$ describe relationships between
blocks of restricted characters.
This technical property plays a key role 
in the proof of the following statement. 
It allows the application of 
Proposition 3.6 of \cite{Spaeth_red_BAW}, and this 
in turn gives the control on the blocks considered. 

In the second part of this section we deduce from the character correspondence 
a generalization of the Harris-Kn\"orr theorem. Since we are 
considering a more general situation, the obtained map 
is only surjective but not bijective in general. 
Furthermore we also loose 
information on the size and structure of defect groups of 
the considered blocks.

\renewcommand{\proofname}{\bf Proof of Theorem \ref{thmA}}
\begin{proof}
Note that since $G=G[b]$, Lemma \ref{KueWatanabe}(c) 
implies $G=NH$.  According to Lemma \ref{KueWatanabe}(a) 
the character $\phi$ is $G$-invariant. Hence 
by the proof of Theorem 4.1 of \cite{NavarroSpaeth} 
there is a group $\wh G$ and a surjective homomorphism 
$\pi: \wh G\rightarrow G$ with cyclic kernel $Z$, 
$p\nmid |Z|$ and the following properties 
\begin{itemize}
\item $\wh N=N_0\times Z$ where $\wh N:=\epsilon^{-1}(N)$, 
$N_0$ is isomorphic to $N$ via $\pi_{N_0}: N_0\rightarrow N$ 
and $N_0\lhd \wh G$. Also the action of $\wh G$ on $N_0$ 
coincides with the action of $H$ on $N$ via $\pi$.
\item The character $\phi_0:=\phi\circ \pi_{N_0}\in\IBr(N_0)$ 
extends to $\wh G$, that is, there exists an extension 
$\w \phi_0\in\IBr(\wh G)$. 
\item The unique irreducible constituent of $(\w\phi_0)_Z$ is faithful
\end{itemize}
(Note that in \cite{NavarroSpaeth} one considers 
an ordinary character but the proof can be transferred 
to Brauer characters thanks to \cite[Theorem (8.14)]{Navarro}.) 
According to the proof the group $\wh G$ coincides with 
$G\times Z$ as set, and its multiplication is defined using a $2$-cocycle 
$\alpha: G/N\times G/N\rightarrow Z$. The group $
N_0$ coincides with $\{(n,1)\mid n\in N\}$.

Let $\wh H:=\{  (h,z) \mid h\in H,\,\, z\in Z\}\leq \wh G$ 
and  $M_0:=\{(m,1)\mid m\in M\}=N_0\cap \wh H$. Then $M_0$ 
is contained in $N_0$ and isomorphic to $M$ via $\pi_{M_0}$. 
Via the isomorphism $\pi_{M_0}$ the block $b'$ defines 
a unique block $b_0'\in\Bl(M_0)$ that satisfies 
$(b'_0)^{N_0}=\bl(\phi_0)$. Then $D_0:=\pi^{-1}(D)\cap N_0$ 
is a defect group of $b'_0$ and satisfies $\Cent_{\wh G}(D_0)\subseteq \wh H$. 
Since $\phi_0$ extends to $\wh G$, 
Theorem \ref{thm1_3}(c)(2) shows that there exist
$\phi'_0 \in\IBr(b'_0)$ and $\wphip_0\in\IBr(\wh H)$ 
 with $(\w\phi'_0)_M=\phi'_0$ and
\[\bl(\wphip_{0,\wh H\cap J})^{J}=\bl(\w\phi_{0,J})
\text{ for every }J \text{ with }\wh N\leq J\leq \wh G.\] 
Note that $\bl(\wphip_{0,\wh H\cap J})$ is admissible 
with respect to $J$ in the sense of \cite[p.212]{Navarro} 
since some defect group of $\bl(\wphip_{0,\wh H\cap J})$ 
contains $D_0$ by \cite[Theorem (9.26)]{Navarro} and 
$\Cent_J(D_0) \subseteq \wh H$ for every $J$ with $N \leq J \leq \wh G$. 

According to \cite[Corollary (8.20)]{Navarro} we have 
\[\IBr(\wh G\mid \phi_0)=
\{ \w\phi_0\eta\mid \eta\in\IBr(\wh G/N_0)\} 
\text{ and }\IBr(\wh H\mid \phi_0')
=\{ \wphip_0\eta\mid \eta\in\IBr(\wh H/M_0)\}.\]
Now Proposition 3.6 of \cite{Spaeth_red_BAW} proves that 
\[ \bl(\wphip_0\eta_H)^{\wh G}=\bl(\w\phi_0\eta) 
\text{ for every }\eta\in\IBr(\wh G)\text{ with } N_0\leq \ker(\eta).\] 
Accordingly the map 
\[\La_0: \IBr(\wh G\mid \phi_0) \longrightarrow 
\IBr(\wh H\mid \phip_0) \text{ given by } 
\w\phi_0\eta\longmapsto \wphip_0\eta_{\wh H} \]
is a bijection such that 
\[ \bl(\La_0(\rho_0))^{\wh G}=\bl(\rho_0) 
\text{ for every } \rho_0 \in \IBr(\wh G\mid \phi_0) .\]
Let $\nu\in\IBr(Z)$. Because of 
$p\nmid |Z|$, $\rho_0 \in \IBr(\wh G\mid \phi_0)$ is contained in 
$\IBr(\wh G\mid \nu)$ if and only if $\bl(\rho_0)$ covers 
$\bl(\nu)$. Now since $Z\leq \Z(\wh G)$, $\bl(\La_0(\rho_0))^{\wh G}$ 
covers $\bl(\nu)$ if and only if $\bl(\La_0(\rho_0))$ 
covers $\bl(\nu)$ according to the definition of induced 
blocks. This implies that $\La_0(\rho_0)\in \IBr(\wh H\mid \nu)$ 
for every $\rho_0 \in \IBr(\wh G\mid \phi_0)$. Altogether we have 
\[ \La_0(\IBr(\wh G\mid \phi_0) 
\cap \IBr(\wh G\mid \nu))\subseteq \IBr(\wh H\mid \nu) \text{ for every } 
\nu\in\IBr(Z).\]
Hence $\La_0$ induces a bijection 
$\La: \IBr( G\mid \phi) \longrightarrow \IBr( H\mid \phip)$. 
According to Proposition 2.4(b) 
of \cite{NavarroSpaeth} the bijection satisfies 
$\bl(\Lambda(\rho))^G=\bl(\rho)$ for every $\rho\in\IBr(G\mid \phi)$.
\end{proof}

We deduce from the correspondence provided by Theorem \ref{thmA}
the following variant of Theorem \ref{thm_gen_HK}.
We start by proving this statement in the case where additionally $G=G[b]$. 

\begin{lem}\label{prop4_5}
Let $N\lhd G$, $H\leq G$ and $M:=N\cap H$. 
Let $b'\in\Bl(M)$ be a block that has a defect group 
$D$ with $\Cent_G(D)\subseteq H$. Assume that $G=G[b]$ for $b:=(b')^N$. 
Then the map $\Bl(H\mid b') \rightarrow \Bl(G\mid b)$ 
given by $B'\mapsto (B')^G$ is surjective. 
\end{lem}

\renewcommand{\proofname}{Proof}
\begin{proof}
Let $\phi\in\IBr(b)$ and $\phip\in\IBr(b')$. 
According to Theorem (9.2) of \cite{Navarro} we have 
\begin{align*}
\Bl(G\mid b)=\{\bl(\rho)\mid \rho \in\IBr(G\mid \phi)\}\text{ and }&
\Bl(H\mid b')=\{\bl(\rho')\mid \rho' \in\IBr(H\mid \phip) \}.
\end{align*}
According to Theorem \ref{thmA} there exists a bijection 
 \[ \Lambda: \IBr(G\mid \phi)\longrightarrow \IBr(H\mid \phip),\]
 such that $\bl(\Lambda(\rho))^G=\bl(\rho)$ for every 
$\rho\in\IBr(G\mid \phi)$. But this implies the statement. 
\end{proof}

Now we use this for the generalization of the Harris-Kn\"orr correspondence. 

\renewcommand{\proofname}{\bf Proof of Theorem \ref{thm_gen_HK}}
\begin{proof}
We prove the statement by induction on $|G/N|$. 
Assume that $G_b\neq G$.

Then by induction there exists 
a surjective map between $\Bl(H\cap G_b\mid b') \rightarrow \Bl(G_b\mid b)$. 
According to Theorem (9.14)(a) of \cite{Navarro} there is a bijection 
between $\Bl(G_b\mid b)$ and $\Bl(G\mid b)$ with $B'\mapsto (B')^G$. 
Since $H_{b'}\leq H\cap G_b$ block induction defines  
a bijection between $\Bl(H\cap G_b\mid b')$ and $\Bl(H\mid b')$, see again Theorem (9.14)(a) of \cite{Navarro}.

We assume in the following that $G_b= G$. 
We know by Lemma \ref{prop4_5} that there exists a surjective map  
$\Bl(H\cap G[b]\mid b') \rightarrow \Bl(G[b]\mid b ) 
$ given by $c\mapsto c^{G[b]}$. 

For any block $C\in\Bl(G\mid b)$ there exists a unique block in $C_0\in\Bl(G[b]\mid b)$
with $C=(C_0)^G$ according to Theorem 3.5 of \cite{DadeBlockExtensions}, see also Theorem 3.5 of \cite{Murai_Dade}.
Further $C_0=(c_0)^{G[b]}$ for some $c_0\in\Bl(H\cap G[b]\mid b')$. 
Any block $c\in\Bl(H\mid c_0)\subseteq \Bl(H\mid b')$ now 
satisfies $c^{G}=C$ according to Lemma \ref{lem2_5}. 
Accordingly the considered map is surjective.
\end{proof}

The following example shows that the considered map 
is not injective in general.

\begin{example}
Let $T$ be the cyclic group of order $63$ generated by an 
element $x$. Let $\sigma$ be the automorphism of $T$ given by 
$x\mapsto x^4$ of order $3$. 
Let $G:=(T\times T)\rtimes \langle\sigma\rangle$,
where the action of $\sigma$ on $T\times T$ is given by $x\times x' \mapsto x^{\sigma}\times (x')^\sigma$.
And let $H:=T\times T$ and $N:= T\times 1 \lhd G$. 
Let $M:=N\cap H = N = T\times 1$, 
and let $b'\in\Bl(M)$ and $b\in\Bl(N)$ be the principal $3$-blocks. 
Then a defect group $D$ of $b'$ is a Sylow $3$-subgroup of $T$. 
The group theoretic assumptions hold in this situation. 
Then $\Bl(H\mid b')$ has seven $3$-blocks. 
On the other hand $\Bl(G\mid b)$ contains three $3$-blocks.
\end{example}

\noindent
{\bf Acknowledgement:} 
A part of this work was done while the first author was visiting
the Department of Mathematics, TU of Kaiserslautern
in October and December 2012. 
He is grateful to Gunter Malle for his kind hospitality.
The authors thank also the referee for several useful comments, especially on 
\cite{DadeBlockExtensions}.

\def\cprime{$'$}


\begin{thebibliography}{Sp{\"a}11b}



\bibitem[Dad73]{DadeBlockExtensions}
E.C.~Dade.
\newblock Block extensions.
\newblock {\em Illinois J. Math.} {\bf 17} (1973), 198--272.
MR0327884 (48\#6226)

\bibitem[Dad84]{Dade84}
E.C.~Dade.
\newblock Extending group modules in a relatively prime case.
\newblock {\em Math. Z.} {\bf 186} (1984), 81--98.
MR735053 (85k:20019)



\bibitem[HK85]{HarrisKnoerr}
M.E.~Harris and R.~Kn{\"o}rr,
\newblock Brauer correspondence for covering blocks of finite groups.
\newblock {\em Comm. Algebra} {\bf 13} (1985), 1213--1218. 
MR780645 (86e:20012)

\bibitem[HK99]{HidaKoshitani}
A.~Hida and S.~Koshitani.
\newblock Morita equivalent blocks in non-normal subgroups 
and $p$-radical blocks in finite groups.
\newblock {\em J.~London Math.~Soc.} {\bf (2)59} (1999), 541--556.
MR1709664 (2000g:20014)

\bibitem[Isa76]{Isa}
I.M.~Isaacs.
\newblock Character Theory of Finite Groups.
\newblock Academic Press, New York, 1976.
MR0460423 (57\#417)

\bibitem[KS13]{KoshitaniSpaeth13}
S.~Koshitani and B.~Sp\"ath.
\newblock The inductive Alperin-McKay and blockwise Alperin Weight
conditions for blocks with cyclic defect groups.
\newblock Preprint 2013.



\bibitem[K{\"u}l90]{Kuelshammer}
B.~K{\"u}lshammer.
\newblock Morita equivalent blocks in {C}lifford theory of finite groups.
\newblock {\em Ast\'erisque} {\bf 181-182} (1990), 209--215.
MR1051250 (91d:20009)
 

 
\bibitem[Mur13]{Murai_Dade}
M.~Murai.
\newblock On blocks of normal subgroups of finite groups.
\newblock {\em Osaka J.~Math.} {\bf 50} (4), 2013  \url{http://www.math.sci.osaka-u.ac.jp/ojm/pdf/3215.pdf}.

\bibitem[Nav98]{Navarro}
G.~Navarro.
\newblock Characters and Blocks of Finite Groups.
\newblock London Math.~Soc.~ Lecture Note Series, {\bf Vol.250}. 
Cambridge University Press, Cambridge, 1998.
MR1632299 (2000a:20018)


\bibitem[NS12]{NavarroSpaeth}
G.~Navarro and B.~Sp\"ath.
\newblock On Brauer's height zero conjecture.
\newblock To appear in {\em J.~European Math.~Soc.}, 2012. 

\bibitem[NT89]{NagaoTsushima}
H.~Nagao and Y.~Tsushima.
\newblock Representations of Finite Groups. 
Transl. from the Japanese.
\newblock {Academic Press, Inc.}, 1989.
MR998775 (90h:20008)


\bibitem[Sp{\"a}11]{Spaeth_AM_red}
B.~Sp{\"a}th.
\newblock A reduction theorem for the Alperin-McKay conjecture.
\newblock {\it  J. reine angew. Math.} {\bf  680}(2013), 153–189.

\bibitem[Sp{\"a}13]{Spaeth_red_BAW}
B.~Sp{\"a}th.
\newblock A reduction theorem for the blockwise Alperin weight conjecture.
\newblock {\em J. Group Theory} {\bf 16} (2013), 159--220.


\end{thebibliography}
\end{document}